\documentclass{article}

\usepackage{amsmath,amssymb,amsthm}

\usepackage{a4wide}

\newtheorem{Theorem}{Theorem}
\newtheorem{Cor}{Corollary}
\newtheorem{Lemma}{Lemma}

\theoremstyle{Remark}
\newtheorem{Remark}{Remark}

\theoremstyle{Definition}
\newtheorem{Definition}{Definition}

\title{Variational problems for H\"olderian functions with free terminal point}

\author{Ricardo Almeida\\
\texttt{ricardo.almeida@ua.pt}
\and Nat\'{a}lia Martins\\
\texttt{natalia@ua.pt}}

\date{\begin{minipage}[t]{1\linewidth}
CIDMA -- Center for Research and Development in Mathematics and Applications,\\
Department of Mathematics, University of Aveiro, Portugal.
\end{minipage}}

\begin{document}

\maketitle

\begin{abstract}
We develop the new variational calculus introduced in 2011 by J. Cresson and I. Greff, where
the classical derivative is substituted by a new complex operator called the scale derivative.
In this paper we consider several nondifferentiable variational problems with free terminal point:
with and without constraints, 
of first and higher-order type.

\bigskip

\noindent \textbf{Keywords:} H\"olderian functions;
calculus of variations; Euler-Lagrange equation;  natural boundary conditions.

\bigskip

\noindent \textbf{Mathematics Subject Classification 2010:}
primary: 39A13; 49K05; 49S05; secondary: 26A27; 26B20; 49K10
\end{abstract}


\section{Introduction}

Dating back to the late 17th century, the calculus of variations has been proved to be a powerfull tool
in several fields, such as physics, geometry, engineering, economics and control theory. One of the earliest variational problems
posed in physics was the problem of determining the shape of a surface of revolution that would encounter the least
resistance when moving through some resisting medium. This problem was solved by Issac Newton and the solution
was published in 1687 in the first book of the collection \emph{Principia}. Since then,
important mathematicians and physicists, such
as John and James Bernoulli, Leibniz, Euler, Lagrange, Jacobi, Weierstrass, Hilbert,
Noether, Tonelly and Lebesgue, studied different variational problems, contributing to the development of the classical  calculus of variations.

The fundamental problem of the classical calculus of variations can be formulated in the following form: among all
differentiable functions $y:[a,b]\rightarrow\mathbb{R}$  such that $y(a) = y_a$ and $y(b) =  y_b$,
where $ y_a,  y_b$ are fixed real numbers,
find the ones that minimize (or maximize) the functional
$$\mathcal{L}[y]=\int_a^b L(t,y(t),y^\prime(t)) dt.$$
It can be proved that the extremizers of this variational problem
must satisfy the
following second-order differential equation
$$\frac{d}{dt}\displaystyle\frac{\partial L}{\partial v}(t,y(t),y^\prime(t))= \displaystyle\frac{\partial L}{\partial y}(t,y(t),y^\prime(t)), \quad \forall t \in [a,b]$$
called the Euler-Lagrange equation
(where $\displaystyle\frac{\partial L}{\partial y}$ and $\displaystyle\frac{\partial L}{\partial v}$
denote, respectively,  the partial derivative of $L$ with respect to the second and third argument).
With the two boundary restrictions, $y(a) = y_a$ and $y(b) =  y_b$, we may determine the extremals of $\mathcal{L}$.
If the boundary condition $y(a) =  y_a$ is not imposed in the initial problem, then in order to find  the extremizers we  have to
add another necessary condition:
$\displaystyle\frac{\partial L}{\partial v}(a,y(a),y^\prime(a))=0$; and thus, instead of the initial condition $y(a)=y_a$, we have one more equation on the system. Also, if $y(b) =  y_b$ is not fixed in the problem, then we need to impose another necessary condition: $\displaystyle\frac{\partial L}{\partial v}(b,y(b),y^\prime(b))=0$.
These two conditions are usually called natural boundary conditions or tranversality conditions.

Since many physical phenomena are described by nondifferentiable functions, and many problems posed in physics can be
formulated using integral functionals, it seemed clear for some mathematicians and physicists that it would be
interesting to develop a calculus of variations for functionals which are defined on a set of nondifferentiable functions.

There exist several different approaches to deal with nondifferentiability in problems of the calculus of variations:
the time scale approach, which typically deals with delta or nabla differentiable functions
\cite{AlmeidaNabla,Bartos,MyID:159,Girejko, malina,MyID:142,NataliaHigherOrderNabla,MyID:198,Malinowska+Martins+Torres2011,Martins+Torres2012}; the fractional approach,
which deals with fractional derivatives of order less than one \cite{Malinowska1,Malinowskabook,Odzijewicz,MyID:182,RicAML,MR2557007,El-Nabulsi,El-Nabulsi:Torres:JMP,withRui:hZ:JDDE,gasta:07};  the quantum approach, which deals with quantum derivatives \cite{Aldwoah,Bang04,Bang05,Malinowska,withMiguel01,Jannussis};
and the scale derivative approach recently introduced by J. Cresson in 2005 \cite{Cresson}.

In this paper we are concerned with the scale derivative approach.
In the paper \cite{CressonGreff}, J. Cresson and I. Greff introduced a new variational calculus
where the classical derivative is replaced by a new complex operator called the scale derivative. Basically, this new derivative
allows the development of an analogue of a differential variational calculus for H\"olderian functions.
The scale variational theory is still
in the very beginning \cite{Cresson,Torres,CressonGreff,Cresson1,Almeida1,Almeida2}, and  much remains to be done.

The main results of the paper \cite{CressonGreff} are:  the Euler-Lagrange equation for a variational problem
with fixed boundary conditions involving a  scale derivative, and the nondifferentiable Noether's Theorem.
In this work we will use the scale derivative introduced in \cite{CressonGreff} (we remark that such derivative is not exactly the same as introduced in \cite{Cresson}).

In \cite{Almeida2} the following scale variational problems were studied:
the isoperimetric problem, the variational problem with dependence on a parameter,
the higher-order variational problem, and the  variational problem with two independent variables, for problems
with fixed boundary conditions.

The main purpose of this paper is to generalize the results of \cite{CressonGreff, Almeida2}.
Let $a$ and $b$ be two fixed real numbers such that $a\leq b$.
We consider  nondifferentiable variational problems

$$\mathcal{I}[y,T]=\int_a^T L\left(t, y(t), \frac{\Box y}{\Box t}(t)\right) \, dt$$

\noindent where $a\leq T\leq b$.  The admissible functions $y$ are H\"olderian such that $y(a)=y_a$,
for some fixed $y_a \in \mathbb{R}$, and  $\Box y/ \Box t$ is the scale derivative of $y$ (to be defined later).
Note that here, in contrast to \cite{CressonGreff, Almeida2},
we have a free terminal point $T$ and no constraint is imposed on $y(T)$. Therefore, $T$ and $y(T)$ become  part of the extremal choice process.

In this paper, for expository convenience, we assume that only the terminal point is variable. It is clear that all the arguments used in the proofs of our results are easily extended to the case of a variable initial point.

This paper is organized as follows. In Section 2 we present the scale derivative as introduced in \cite{CressonGreff}
and briefly review some of its properties, namely, the Leibniz rule and the Barrow rule in the scale calculus context.
Our main results are presented in Section 3. In Subsection 3.1 we prove necessary
and sufficient conditions to obtain extremals for complex valued integral functionals with various type of constraint (imposed over the terminal point $T$  and/or over $y(T)$).
In Subsection 3.2 we briefly show how higher-order scale derivatives can be included on the variational problem (Theorem \ref{E-L+HTC}).


\section{Preliminaries}

To make the paper self-contained, we begin with a review of the definitions and results from \cite{CressonGreff} needed to the present work.

In what follows we consider $\alpha, \beta \in \  ]0,1[$, $h \in \ ]0,1[$ with $h \ll 1$, $a,b\in \mathbb{R}$ with $a<b$, and define the interval $I:=[a-h,b+h]$.

\begin{Definition} Let $f:I\rightarrow \mathbb{R}$. The $h$-forward difference operator is defined by
$$\Delta_h[f](t):=\frac{f(t+ h)-f(t)}{h}, \quad t\in[a-h,b].$$
The $h$-backward difference operator is defined by
$$\nabla_h[f](t):=\frac{f(t)-f(t-h)}{h}, \quad t\in[a,b+h].$$

\end{Definition}

\begin{Remark}
Obviously, if $f$ is a differentiable  function (in the classical sense),
then
$$\lim_{h\to 0}\Delta_{h}[f](t)=\lim_{h\to 0}\nabla_{h}[f](t) = f'(t).$$
\end{Remark}

\begin{Definition}
The $h$-scale derivative of $f:I \rightarrow \mathbb{R}$   is defined by
\begin{equation}
\label{eq:def:cp}
\frac{{\Box_{h}}f}{\Box t}(t)
=\frac12 \bigg [ \bigg( \Delta_{h}[f](t) + \nabla_{h}[f](t) \bigg)
+i \bigg( \Delta_{h}[f](t) - \nabla_{h}[f](t) \bigg) \bigg ], \quad t \in [a,b], \quad i^2=-1.
\end{equation}
For complex valued functions $f$ we define
$$
\frac{{\Box_{h}}f}{\Box t}(t)
= \frac{{\Box_{h}}\mbox{Re}f}{\Box t}(t)
+ i \frac{{\Box_{h}}\mbox{Im}f}{\Box t}(t)
$$
where Re$f$ and Im$f$ denote, respectively, the real and imaginary part of $f$.
\end{Definition}

\begin{Remark}
Again, if  $f:I \rightarrow \mathbb{R}$ is a differentiable  function, then
$$\lim_{h\to 0} \frac{{\Box_{h}}f}{\Box t}(t)= f'(t).$$
\end{Remark}

Let ${C^0_{conv}}([a,b]\times ]0,1[,\mathbb{C})$ be the subspace of ${C^0}([a,b]\times ]0,1[,\mathbb{C})$
such that for any function $g \in {C^0_{conv}}([a,b]\times ]0,1[,\mathbb{C})$, the limit
$$\lim_{h\to 0}g(t,h)$$
exists for any $t\in [a,b]$. Denote by $E$ a complementary space of ${C^0_{conv}}([a,b]\times ]0,1[,\mathbb{C})$
in ${C^0}([a,b]\times  ]0,1[,\mathbb{C})$ and
denote by  $\pi$ the projection of ${C^0_{conv}}([a,b]\times ]0,1[,\mathbb{C})\oplus E $ onto ${C^0_{conv}}([a,b]\times ]0,1[,\mathbb{C})$, that is,

$$
\begin{array}{lcll}
\pi: & {C^0_{conv}}([a,b]\times ]0,1[,\mathbb{C})
\oplus E & \to & {C^0_{conv}}([a,b]\times ]0,1[,\mathbb{C})\\
&g:= g_{conv}+g_E  & \mapsto & \pi(g)=g_{conv}.
\end{array}
$$

In order to define the scale derivative, we need to introduce the following operator:

$$
\begin{array}{lcll}
\left< \cdot \right>: & {C^0}([a,b]\times ]0,1[,\mathbb{C}) & \to & \mathcal{F}([a,b],\mathbb{C})\\
& g & \mapsto & \left< g \right>: t\mapsto
\displaystyle\lim_{h\to 0}\pi(g)(t,h)
\end{array}
$$
where $\mathcal{F}([a,b],\mathbb{C})$ denotes the set of functions $f:[a,b]\rightarrow \mathbb{C}$.

Finally, we arrive to the main concept introduced in \cite{CressonGreff}:
the scale derivative of $f$ (without the dependence of $h$).

\begin{Definition}
\label{def:ourHD}
The scale derivative of $f\in {C^0}(I,\mathbb{C})$, denoted by $\frac{\Box f}{\Box t}$,  is defined by
\begin{equation}
\label{eq:scaleDer}
\frac{\Box f}{\Box t}(t):=\left< \frac{{\Box_{h}}f}{\Box t} \right>(t), \quad t \in [a,b].
\end{equation}
\end{Definition}

\begin{Remark} As remarked in {\rm \cite{CressonGreff}}, the operator $\left< \cdot \right>$ depends on the choice of the complementary space $E$.
However, this dependence does not change the form of the properties of this derivative.
\end{Remark}

From now on, we assume that a complementary space $E$ is fixed.

It is clear that the scale derivative \eqref{eq:scaleDer} is a linear operator and the scale derivative of a constant is zero.

\begin{Remark}
It is clear that if $f:I \rightarrow \mathbb{R}$ is a differentiable function, then $\frac{\Box f}{\Box t}(t)= f'(t).$
\end{Remark}

Higher-order scale derivatives can be defined as usual.
For a given $n\in \mathbb{N}$,  denote by $I^n:=[a-nh,b+nh]$.
The $n$\emph{th}  scale derivative of a function $f: I^n \rightarrow \mathbb{C}$ is the function
$\frac{\Box^n f}{\Box t^n}$ defined recursively by
$$\frac{\Box^n f}{\Box t^n}(t)=\frac{\Box}{\Box t}\left( \frac{\Box^{n-1} f}{\Box t^{n-1}} \right)(t), \quad t\in[a,b],$$
where $\frac{\Box^1 f}{\Box t^1}:= \frac{\Box f}{\Box t}$ and $\frac{\Box^0 f}{\Box t^0}:=f$.

In what follows we will denote
$$C^n_\Box([a,b], \mathbb{K}):= \{f \in C^0(I^n, \mathbb{K})\mid \frac{\Box^k f}{\Box t^k}\in C^0(I^{n-k}, \mathbb{C}), k=1,2,\ldots,n \}, \quad  \mathbb{K}=
 \mathbb{R} \mbox{ or }  \mathbb{K}= \mathbb{C}.$$

It can be proved that the scale derivative  satisfies a scale analogue of
Leibniz's and Barrow's rule in certain subclasses of $C^0(I,\mathbb{C})$.

First let us recall the definition of H\"{o}lderian function.

\begin{Definition}
Let $f\in C^0(I,\mathbb{C})$. We say that $f$ is H\"{o}lderian of H\"{o}lder exponent
$\alpha$ if there exists a constant $C>0$ such that,
for all $s,t \in I$, the inequality
$$
|f(t)-f(s)|\leq C |t-s|^\alpha
$$
holds. The set of  H\"{o}lderian functions of H\"{o}lder
exponent $\alpha$ defined on $I$ is denoted by $H^\alpha(I, \mathbb{C})$.
\end{Definition}

\begin{Theorem}[The scale Leibniz rule \cite{CressonGreff}]
\label{LeibnizRule}
Let $\alpha, \beta \in ]0,1[$ be such that $\alpha+\beta>1$.
For $f\in H^\alpha(I, \mathbb{R})$ and $g\in H^\beta(I, \mathbb{R})$, we have
\begin{equation*}
\frac{\Box (f.g)}{\Box t}(t)
=\frac{\Box f}{\Box t}(t).g(t)+f(t).\frac{\Box g}{\Box t}(t), \quad t \in [a,b].
\end{equation*}
\end{Theorem}

\begin{Theorem}[The scale Barrow rule \cite{CressonGreff}]
\label{Barrow}
Let $f\in C^1_{\Box}([a,b],\mathbb{R})$ be such that
\begin{equation}
\label{nec_condition}
\lim_{h\to 0} \int_a^b \left(\frac{\Box_h f}{\Box t}\right)_E(t) \, dt=0,
\end{equation}
where
$ \frac{\Box_h f}{\Box t}:=  \left(\frac{\Box_h f}{\Box t}\right)_{conv} +  \left(\frac{\Box_h f}{\Box t}\right)_E.$
Then,
$$
\int_a^b \frac{\Box f}{\Box t}(t)\, dt=f(b)-f(a).
$$
\end{Theorem}

\begin{Remark}
Although Theorems \ref{LeibnizRule} and \ref{Barrow} were proven for real valued functions,
they still hold for complex valued functions (the proofs are similar to the ones given in {\rm \cite{CressonGreff})}.
\end{Remark}

The following result is an easy consequence of the complex versions of Theorems \ref{LeibnizRule} and \ref{Barrow}.

\begin{Theorem}[The scale integration by parts formula] \label{integration by parts}
Let $f \in H^{\alpha}(I,\mathbb{C})$ and $g \in H^{\beta}(I,\mathbb{C})$ be such that
$\alpha+\beta>1$  and
$$
\lim_{h\to 0} \int_a^b \left(\frac{\Box_h (f \cdot g)}{\Box t}\right)_E(t) \, dt=0.
$$
Then
$$
\int_a^b \frac{\Box f}{\Box t}(t) \cdot g(t) dt = \left[f(t)g(t)\right]_a^b - \int_a^b f(t)\cdot \frac{\Box g}{\Box t}(t) dt.
$$
\end{Theorem}

Next we prove the scale Taylor Theorem of first order.

\begin{Theorem}\label{Taylor}
Given $f\in C^2_{\Box}([a,b],\mathbb{C})$ and $t\in [a,b]$ such that
$\frac{\Box f}{\Box t} \in H^{\alpha}(I,\mathbb{C})$,
$$\lim_{h\to 0} \int_a^t \left(\frac{\Box_h f}{\Box \tau}\right)_E(\tau) \, d\tau=0,$$
and
$$
\lim_{h\to 0} \int_a^t \left(\frac{\Box_h ( \frac{\Box f}{\Box \tau}\cdot(t-\tau))}{\Box \tau}\right)_E(\tau) \, d\tau=0,
$$

\noindent then
$$f(t)=f(a)+\frac{\Box f}{\Box t}(a)(t-a)+O(t-a)^2.$$
\end{Theorem}
\begin{proof} First observe that $\tau\mapsto t-\tau$ is of class $C^1$ and thus
$$\frac{\Box }{\Box \tau}(t-\tau) =\frac{d}{d\tau}(t-\tau)=-1.$$
Using Theorem \ref{Barrow} we conclude that
$$\begin{array}{ll}
f(t)-f(a) & =\displaystyle \int_a^t \frac{\Box f}{\Box \tau}(\tau) \, d \tau\\
         &  =-\displaystyle \int_a^t \frac{\Box f}{\Box \tau}(\tau)\cdot  \frac{\Box }{\Box \tau}(t-\tau) \, d \tau.\\
\end{array}$$
Using the scale integration by parts formula (Theorem \ref{integration by parts}), we get
$$\begin{array}{ll}
f(t)-f(a) & = \displaystyle\frac{\Box f}{\Box t}(a)(t-a)+ \int_a^t \frac{\Box^2 f}{\Box \tau^2}(\tau)\cdot  (t-\tau) \, d \tau\\
& = \displaystyle\frac{\Box f}{\Box t}(a)(t-a)+ M  (t-a)^2,\\
\end{array}$$
for some $M\in\mathbb C$, which ends the proof.
\end{proof}
\vspace{0.3 cm}

We finish this section with the following useful result.

\begin{Lemma}\label{Lemma1} Let $\alpha, \beta \in ]0,1[$ be such that $\alpha\leq \beta$. Then, for all $y \in H^{\alpha}(I,\mathbb{C})$
and $\eta \in H^{\beta}(I,\mathbb{C})$, $y+\eta \in  H^{\alpha}(I,\mathbb{C})$.
\end{Lemma}


\section{Main results}

The aim is to exhibit several necessary and sufficient conditions in order to determine scale extremals for a certain class of functionals, which domain is the set of H\"{o}lderian functions. Not only the Euler-Lagrange equation will be obtained, but also natural boundary conditions will appear.

\subsection{First-order variational problems}

We consider the following functional:

$$\mathcal{I}[y,T]=\displaystyle \int_a^T L\left(t, y(t), \frac{\Box y}{\Box t}(t)\right) \, dt$$

\noindent defined on

$$\mathcal{A}= \{y \in H^\alpha (I, \mathbb{R}): y(a)=y_a \wedge y \in C^1_{\Box}([a,b],\mathbb{R})\},$$

\noindent where $T\in \mathbb{R}$ is such that $a\leq T\leq b$, the Lagrangian
$L=L(t,y,v):[a,b]\times \mathbb{R} \times \mathbb{C}  \rightarrow \mathbb{C}$ is of class $C^1$,
and $y_a \in \mathbb{R}$ is a given fixed real number.
We emphasize that we have a free terminal point $T$ and no constraint on $y(T)$. Hence,  $T$ and
$y(T)$ become part of the extremal choice process.

\begin{Definition}\label{scale extremal}
We say that $(y,T)$ is a scale extremal of functional $\mathcal{I}$ defined on $\mathcal{A}$
if, for any $\eta \in H^\beta(I, \mathbb{R})\cap C^1_{\Box}([a,b], \mathbb{R})$ such that $\eta(a)=0$, and $\delta \in \mathbb{R}$,
$$
\frac{d}{d\varepsilon}\left.\mathcal{I}[y + \varepsilon \eta, T + \varepsilon \delta ]\right|_{\varepsilon = 0}=0.
$$
\end{Definition}

\vspace{0.3 cm}

Our first goal is to obtain a necessary and a sufficient condition to $(y,T)$ be a scale extremal of the following functional

\begin{equation}%
\left\{
\begin{array}
[c]{l}%
\mathcal{I}[y,T]=\displaystyle \int_a^T L\left(t, y(t), \frac{\Box y}{\Box t}(t)\right) \, dt\\
\\
y\in \mathcal{A}, \, T\in[a,b].
\end{array}
\tag{P}\label{P}%
\right.
\end{equation}

In the sequel we assume that $ \alpha + \beta >1$ and $\alpha \leq \beta$.
For simplicity of notation, we introduce the operator $[\cdot]$ defined by
$$[y](t) =\left(t,y(t),\frac{\Box y}{\Box t}(t)\right).$$

\begin{Theorem}[The scale Euler--Lagrange equation and natural boundary conditions I]
\label{E-L+TC}
Let $\widetilde{T}\in [a,b]$
and $\widetilde{y} \in \mathcal{A}$ be such that $\displaystyle\frac{\partial L}{\partial v}[\widetilde{y}] \in H^\alpha(I,\mathbb{C})$ and
$$
\lim_{h \to 0} \int_a^{\widetilde{T}}\left( \frac{\Box_h}{\Box t}\left(\frac{\partial L}{\partial v}[\widetilde{y}]\cdot \eta  \right)(t) \right)_E \, dt =0
$$
for all $\eta \in H^\beta(I,\mathbb{R}) \cap C^1_{\Box}([a,b], \mathbb{R})$ such that $\eta(a)=0$.
The pair $(\widetilde{y},\widetilde{T})$ is a scale extremal of functional $\mathcal{I}$ defined on $\mathcal{A}$
if and only if the following conditions hold:
\begin{enumerate}
\item $\displaystyle\frac{\partial L}{\partial y}[\widetilde{y}](t)
=\frac{\Box}{\Box t}\left(\frac{\partial L}{\partial v}[\widetilde{y}]\right)(t)$ for all $t\in [a,\widetilde{T}]$;
\item $\displaystyle\frac{\partial L}{\partial v}\left( \widetilde{T},\widetilde{y}(\widetilde{T}),\frac{\Box \widetilde{y}}{\Box t}(\widetilde{T}) \right)=0$;
\item $ \displaystyle L\left( \widetilde{T},\widetilde{y}(\widetilde{T}),\frac{\Box \widetilde{y}}{\Box t}(\widetilde{T}) \right) =0$.
\end{enumerate}
\end{Theorem}

\begin{proof}
Suppose that $(\widetilde{y},\widetilde{T})$ is a scale extremal of problem \eqref{P}. Hence, by definition,
$$
\frac{d}{d\varepsilon}\left.\mathcal{I}[\widetilde{y} + \varepsilon \eta, \widetilde{T} + \varepsilon \delta ]\right|_{\varepsilon = 0}=0
$$
for all $\eta \in H^\beta(I,\mathbb{R}) \cap C^1_{\Box}([a,b], \mathbb{R})$ such that $\eta (a)=0$, and all $\delta \in \mathbb{R}$.

Note that, by Lemma \ref{Lemma1}, $\widetilde{y} + \varepsilon \eta \in \mathcal{A}$. Also note that

\begin{equation}
\begin{split} \label{provaE-L0}
0& =  \frac{d}{d\varepsilon}\left.\mathcal{I}[\widetilde{y} + \varepsilon \eta, \widetilde{T} + \varepsilon \delta ]\right|_{\varepsilon = 0} \\
&= \frac{d}{d\varepsilon} \left(\int_a^{\widetilde{T}+ \varepsilon \delta} L\left(t, \widetilde{y}(t) + \varepsilon \eta(t), \frac{\Box \widetilde{y}}{\Box t}(t) + \varepsilon \frac{\Box \eta}{\Box t}(t)  \right) \, dt\right)\Big|_{\varepsilon = 0}\\
& = \displaystyle \int_a^{\widetilde{T}} \left[ \frac{\partial L}{\partial y}[\widetilde{y}](t)\cdot \eta(t)
+\frac{\partial L}{\partial v}[\widetilde{y}](t)\cdot \frac{\Box \eta}{\Box t}(t) \right]\, dt +
\displaystyle L\left( \widetilde{T},\widetilde{y}(\widetilde{T}),\frac{\Box \widetilde{y}}{\Box t}(\widetilde{T}) \right)\cdot \delta \\
& = \displaystyle \int_a^{\widetilde{T}} \left[ \frac{\partial L}{\partial y}[\widetilde{y}](t) -  \frac{\Box }{\Box t}\left(\frac{\partial L}{\partial v}[\widetilde{y}]\right)(t)\right]\cdot \eta(t) \, dt
+ \displaystyle \left[ \frac{\partial L}{\partial v}[\widetilde{y}](t)\cdot \eta(t)\right]_a^{\widetilde{T}}
+  \displaystyle L\left( \widetilde{T},\widetilde{y}(\widetilde{T}),\frac{\Box \widetilde{y}}{\Box t}(\widetilde{T}) \right)\cdot \delta
\end{split}
\end{equation}

Since $\eta(a)=0$, then
\begin{equation} \label{provaE-L}
\displaystyle \int_a^{\widetilde{T}} \left[ \frac{\partial L}{\partial y}[\widetilde{y}](t) -  \frac{\Box }{\Box t}\left(\frac{\partial L}{\partial v}[\widetilde{y}]\right)(t)\right]\cdot \eta(t) \, dt
+ \displaystyle \frac{\partial L}{\partial v}[\widetilde{y}](\widetilde{T})\cdot \eta(\widetilde{T})
+  \displaystyle L\left( \widetilde{T},\widetilde{y}(\widetilde{T}),\frac{\Box \widetilde{y}}{\Box t}(\widetilde{T}) \right)\cdot \delta =0
\end{equation}

\noindent If we restrict the variations in (\ref{provaE-L}) to those such that $\eta(\widetilde{T})=0$ and  $\delta=0$, we get
$$
 \displaystyle \int_a^{\widetilde{T}} \left[ \frac{\partial L}{\partial y}[\widetilde{y}](t) -  \frac{\Box }{\Box t}\left(\frac{\partial L}{\partial v}[\widetilde{y}]\right)(t)\right]\cdot \eta(t) \, dt=0.
$$
From the fundamental lemma of the calculus of variations it follows that
\begin{equation}\label{E-L equation}
\frac{\partial L}{\partial y}[\widetilde{y}](t) -  \frac{\Box }{\Box t}\left(\frac{\partial L}{\partial v}[\widetilde{y}]\right)(t)=0
\end{equation}
for all $t\in [a,\widetilde{T}]$.

\noindent If we restrict in (\ref{provaE-L})  to those $\eta$ such that $\eta(\widetilde{T})=0$,  we get
\begin{equation}\label{exp1}
\displaystyle \int_a^{\widetilde{T}} \left[ \frac{\partial L}{\partial y}[\widetilde{y}](t) -  \frac{\Box }{\Box t}\left(\frac{\partial L}{\partial v}[\widetilde{y}]\right)(t)\right]\cdot \eta(t) \, dt
+  \displaystyle L\left( \widetilde{T},\widetilde{y}(\widetilde{T}),\frac{\Box \widetilde{y}}{\Box t}(\widetilde{T}) \right)\cdot \delta=0.
\end{equation}

\noindent Substituting the scale Euler-Lagrange equation (\ref{E-L equation}) into (\ref{exp1}) we obtain
$$L\left( \widetilde{T},\widetilde{y}(\widetilde{T}),\frac{\Box \widetilde{y}}{\Box t}(\widetilde{T}) \right)\cdot \delta=0.$$
By the arbitrariness of $\delta$ we get
$$L\left( \widetilde{T},\widetilde{y}(\widetilde{T}),\frac{\Box \widetilde{y}}{\Box t}(\widetilde{T}) \right)=0.$$

\noindent Substituting  $\delta=0$ and  the scale Euler-Lagrange equation (\ref{E-L equation}) into (\ref{provaE-L}), we have that

$$
\displaystyle \frac{\partial L}{\partial v}[\widetilde{y}](\widetilde{T})\cdot \eta(\widetilde{T})=0.
$$
From the arbitrariness of $\eta$, we conclude that
$$
\displaystyle\frac{\partial L}{\partial v}\left( \widetilde{T},\widetilde{y}(\widetilde{T}),\frac{\Box \widetilde{y}}{\Box t}(\widetilde{T}) \right)=0.
$$
\end{proof}

\begin{Cor}{\rm (\cite{CressonGreff})} Suppose that $T$ is fixed in problem \eqref{P} and the set of admissible functions is given by
$$\mathcal{B}:= \mathcal{A} \cap \{y \in H^\alpha (I, \mathbb{R}): y(T)=y_T\}$$
for some fixed real number $y_T$.
Let
$\widetilde{y} \in \mathcal{B}$ be such that $\displaystyle\frac{\partial L}{\partial v}[\widetilde{y}] \in H^\alpha(I,\mathbb{C})$ and
$$
\lim_{h \to 0} \int_a^{T}\left( \frac{\Box_h}{\Box t}\left(\frac{\partial L}{\partial v}[\widetilde{y}]\cdot \eta  \right)(t) \right)_E \, dt =0
$$
for all $\eta \in H^\beta(I,\mathbb{R}) \cap C^1_{\Box}([a,b], \mathbb{R})$ such that $\eta(a)=\eta(T)=0$.
Then $\widetilde{y}$ is a scale extremal of functional $\mathcal{I}$   in the class
$\mathcal{B}$
if and only if $$\displaystyle\frac{\partial L}{\partial y}[\widetilde{y}](t)
=\frac{\Box}{\Box t}\left(\frac{\partial L}{\partial v}[\widetilde{y}]\right)(t), \quad \quad \forall t\in [a,T].$$
\end{Cor}

\begin{Remark} {\rm  (cf. \cite{Chiang})} If we restrict the set of admissible functions to be the set $\{y \in C^1([a,b], \mathbb{R}): y(a)=y_a\}$,
then problem \eqref{P} reduces to the classical variational problem

\begin{equation}%
\left\{
\begin{array}
[c]{l}%
\mathcal{I}[y,T]=\displaystyle \int_a^T L\left(t, y(t), y^\prime(t)\right) \, dt \quad \rightarrow extremize\\
\\
y\in C^1 ([a,b], \mathbb{R})  \\
\\
y\left(  a\right)  =y_a,
\end{array}
\right.
\label{ProblemaClassico}
\end{equation}
and, by Theorem \ref{E-L+TC}, we can conclude that if $(\widetilde{y}, \widetilde{T})$ is an extremizer (that is, minimizer or maximizer)
of problem \eqref{ProblemaClassico}, then
\begin{enumerate}
\item $\displaystyle\frac{\partial L}{\partial y}[\widetilde{y}](t)
=\frac{d}{dt}\left(\frac{\partial L}{\partial v}[\widetilde{y}]\right)(t)$ for all $t\in [a,\widetilde{T}]$;
\item $\displaystyle\frac{\partial L}{\partial v}\left( \widetilde{T},\widetilde{y}(\widetilde{T}),\widetilde{y}^{\prime}(\widetilde{T}) \right)=0$;
\item $ \displaystyle L\left( \widetilde{T},\widetilde{y}(\widetilde{T}),\widetilde{y}^{\prime}(\widetilde{T}) \right) =0$.
\end{enumerate}

\end{Remark}

Doing similar calculations as done in the proof of Theorem \ref{E-L+TC} one can prove the following result.

%

\begin{Theorem}[The scale Euler--Lagrange equation and natural boundary conditions II]
\label{E-L+TC 2} Let $\widetilde{T}\in [a,b]$
and $\widetilde{y} \in C^1_{\Box}([a,b], \mathbb{R})$ be such that $\displaystyle\frac{\partial L}{\partial v}[\widetilde{y}] \in H^\alpha(I,\mathbb{C})$ and
$$
\lim_{h \to 0} \int_a^{\widetilde{T}}\left( \frac{\Box_h}{\Box t}\left(\frac{\partial L}{\partial v}[\widetilde{y}]\cdot \eta  \right)(t) \right)_E \, dt =0
$$
for all $\eta \in H^\beta(I,\mathbb{R}) \cap C^1_{\Box}([a,b], \mathbb{R})$.
The pair $(\widetilde{y},\widetilde{T})$ is a scale extremal of functional $\mathcal{I}$ defined on $C^1_{\Box}([a,b], \mathbb{R})$
if and only if the following conditions hold:\\
\begin{enumerate}
\item  $\displaystyle\frac{\partial L}{\partial y}[\widetilde{y}](t)
=\frac{\Box}{\Box t}\left(\frac{\partial L}{\partial v}[\widetilde{y}]\right)(t)$ for all $t\in [a,\widetilde{T}]$;\\
\item $\displaystyle\frac{\partial L}{\partial v}\left( a,\widetilde{y}(a),\frac{\Box \widetilde{y}}{\Box t}(a) \right)=0$;\\
\item $\displaystyle\frac{\partial L}{\partial v}\left( \widetilde{T},\widetilde{y}(\widetilde{T}),\frac{\Box \widetilde{y}}{\Box t}(\widetilde{T}) \right)=0$;\\
\item $\displaystyle L\left( \widetilde{T},\widetilde{y}(\widetilde{T}),\frac{\Box \widetilde{y}}{\Box t}(\widetilde{T}) \right) =0$.
 \end{enumerate}
\end{Theorem}

\begin{Cor}{\rm (\cite{Almeida2})} Suppose that $T$ is fixed in problem \eqref{P} and that
the boundary conditions $y(a)=y_a$ and $y(T)=y_T$ are not present.
Let $\mathcal{D}= \{y \in H^\alpha (I, \mathbb{R}):  y \in C^1_{\Box}([a,b], \mathbb{C})\},$
and  $\widetilde{y} \in \mathcal{D}$ be such that $\displaystyle\frac{\partial L}{\partial v}[\widetilde{y}] \in H^\alpha(I,\mathbb{C})$, and
$$
\lim_{h \to 0} \int_a^{T}\left( \frac{\Box_h}{\Box t}\left(\frac{\partial L}{\partial v}[\widetilde{y}]\cdot \eta  \right)(t) \right)_E \, dt =0
$$
for all $\eta \in H^\beta(I,\mathbb{R}) \cap C^1_{\Box}([a,b], \mathbb{R})$.
Then $\widetilde{y}$ is a scale extremal of functional $\mathcal{I}$ in the class $\mathcal{D}$
if and only if the following conditions hold:
\begin{enumerate}
\item $\displaystyle\frac{\partial L}{\partial y}[\widetilde{y}](t)
=\frac{\Box}{\Box t}\left(\frac{\partial L}{\partial v}[\widetilde{y}]\right)(t)$ for all $t\in [a,T]$;
\item $\displaystyle\frac{\partial L}{\partial v}\left( a,\widetilde{y}(a),\frac{\Box \widetilde{y}}{\Box t}(a) \right)=0$;
\item $\displaystyle\frac{\partial L}{\partial v}\left( T,\widetilde{y}(T),\frac{\Box \widetilde{y}}{\Box t}(T) \right)=0$.
\end{enumerate}
\end{Cor}

\vspace{0.3 cm}

In the proof of Theorem \ref{E-L+TC} we proved that  $(\widetilde{y},\widetilde{T})$ is a scale extremal of problem  (\ref{P}), if and only if,
for arbitrary $\eta \in H^\beta(I,\mathbb{R}) \cap C^1_{\Box}([a,b], \mathbb{R})$ such that $\eta(a)=0$,  and $\delta \in \mathbb{R}$,

\begin{equation}
\label{*}
\displaystyle \int_a^{\widetilde{T}} \left[ \frac{\partial L}{\partial y}[\widetilde{y}](t) -  \frac{\Box }{\Box t}\left(\frac{\partial L}{\partial v}[\widetilde{y}]\right)(t)\right]\cdot \eta(t) \, dt
+ \displaystyle \frac{\partial L}{\partial v}[\widetilde{y}](\widetilde{T})\cdot \eta(\widetilde{T})
+  \displaystyle L\left( \widetilde{T},\widetilde{y}(\widetilde{T}),\frac{\Box \widetilde{y}}{\Box t}(\widetilde{T}) \right)\cdot \delta =0.
\end{equation}

In what follows we will write $\eta(\widetilde{T})$ in terms of $\delta$, the increment over time, and the increment over space,
$$
\delta\widetilde{y}_{\widetilde{T}}:=(\widetilde{y}+\eta)(\widetilde{T}+\delta)-\widetilde{y}(\widetilde{T}).
$$

Assuming that $\widetilde{y}, \eta \in C^2_{\Box}([a,b],\mathbb{R})$ and for those $\eta$ such that $\frac{\Box \eta}{\Box t}(\widetilde{T})=0$, and $|\delta| \ll 1$, then, by Theorem \ref{Taylor}, we deduce that
$$(\widetilde{y}+\eta)(\widetilde{T}+\delta)-(\widetilde{y}+\eta)(\widetilde{T})=\frac{\Box \widetilde{y}}{\Box t}(\widetilde{T})\cdot \delta+O(\delta^2)$$
and so we obtain the formula for the increment over space
$$\delta\widetilde{y}_{\widetilde{T}}=\frac{\Box \widetilde{y}}{\Box t}(\widetilde{T})\cdot \delta+\eta(\widetilde{T})+O(\delta^2)$$
which is equivalent to
\begin{equation}
\label{transv1}\eta(\widetilde{T})=\delta\widetilde{y}_{\widetilde{T}}-\frac{\Box \widetilde{y}}{\Box t}(\widetilde{T})\cdot \delta+O(\delta^2).
\end{equation}

Substituting (\ref{E-L equation}) and (\ref{transv1}) into  (\ref{*}) we get
 \begin{equation}\label{transv2}\delta\left[ L\left( \widetilde{T},\widetilde{y}(\widetilde{T}),\frac{\Box \widetilde{y}}{\Box t}(\widetilde{T}) \right)-  \frac{\partial L}{\partial v}[\widetilde{y}](\widetilde{T})\cdot \frac{\Box \widetilde{y}}{\Box t}(\widetilde{T})\right]
 +\delta\widetilde{y}_{\widetilde{T}}\cdot\frac{\partial L}{\partial v}[\widetilde{y}](\widetilde{T})+ O(\delta^2) =0.
\end{equation}

Depending on the constraints that may be imposed over the terminal point $T$ and/or over the boundary condition $y(T)$, several natural boundary conditions can be obtained from condition (\ref{transv2}). Obviously, if they are both fixed we do not get any extra condition (see \cite{CressonGreff}).

Next we consider the case where the boundary conditions $y(a)$ and $y(T)$ are fixed and $T$ is free.

\begin{Theorem}\label{E-L+TC III}
[The scale Euler--Lagrange equation and natural boundary conditions III]
In the conditions of Theorem \ref{E-L+TC},
the pair $(\widetilde{y},\widetilde{T})$ is a scale extremal of functional $\mathcal{I}$ defined on
$$\{y \in H^\alpha (I^2, \mathbb{R}): y(a)=y_a \wedge y(T)=y_T \wedge y \in C^2_{\Box}([a,b],\mathbb{R})\},$$
where $y_a, y_T \in \mathbb{R}$ are fixed real numbers,
if and only if the following conditions hold:\\

\begin{enumerate}
\item $\displaystyle\frac{\partial L}{\partial y}[\widetilde{y}](t)
=\frac{\Box}{\Box t}\left(\frac{\partial L}{\partial v}[\widetilde{y}]\right)(t)$ for all $t\in [a,\widetilde{T}]$;\\

\item $\displaystyle L\left( \widetilde{T},\widetilde{y}(\widetilde{T}),\frac{\Box \widetilde{y}}{\Box t}(\widetilde{T}) \right) =
\frac{\partial L}{\partial v}[\widetilde{y}](\widetilde{T})\cdot \frac{\Box \widetilde{y}}{\Box t}(\widetilde{T}).$

\end{enumerate}
\end{Theorem}

\begin{proof}
In this case, $\delta\widetilde{y}_{\widetilde{T}}=0$ and $\delta$ is arbitrary. Thus from  (\ref{transv2}) we obtain the condition
$$ L\left( \widetilde{T},\widetilde{y}(\widetilde{T}),\frac{\Box \widetilde{y}}{\Box t}(\widetilde{T}) \right)=\frac{\partial L}{\partial v}[\widetilde{y}](\widetilde{T})\cdot \frac{\Box \widetilde{y}}{\Box t}(\widetilde{T}).$$

\end{proof}

Now we shall generalize Theorem \ref{E-L+TC III} considering the case where we have the boundary condition $\widetilde{y}(T)=\psi(T)$,
where $\psi$ is a given function of
class $C^2_{\Box}([a,b],\mathbb{R})$, and $T$ is free.

\begin{Theorem}
[The scale Euler--Lagrange equation and natural boundary conditions IV]
Let $\psi$ be a given function of class $C^2_{\Box}([a,b],\mathbb{R})$.
In the conditions of Theorem \ref{E-L+TC},
the pair $(\widetilde{y},\widetilde{T})$ is a scale extremal of functional $\mathcal{I}$ defined on
$$\{y \in H^\alpha (I^2, \mathbb{R}): y(a)=y_a \wedge y(T)=\psi(T) \wedge y \in C^2_{\Box}([a,b],\mathbb{R})\}$$
if and only if the following conditions hold:\\

\begin{enumerate}
\item $\displaystyle\frac{\partial L}{\partial y}[\widetilde{y}](t)
=\frac{\Box}{\Box t}\left(\frac{\partial L}{\partial v}[\widetilde{y}]\right)(t)$ for all $t\in [a,\widetilde{T}]$;\\

\item $\displaystyle L\left( \widetilde{T},\widetilde{y}(\widetilde{T}),\frac{\Box \widetilde{y}}{\Box t}(\widetilde{T}) \right) =
\frac{\partial L}{\partial v}[\widetilde{y}](\widetilde{T})\cdot \left(\frac{\Box \widetilde{y}}{\Box t}(\widetilde{T})-\frac{\Box \psi}{\Box t}(\widetilde{T})\right).$
\end{enumerate}
\end{Theorem}

\begin{proof}
From Theorem \ref{Taylor} we can conclude that
\begin{equation}\label{transv3}
\begin{array}{ll}
\delta\widetilde{y}_{\widetilde{T}}&=\displaystyle\psi(\widetilde{T}+\delta)-\psi(\widetilde{T})\\
&\displaystyle =\frac{\Box \psi}{\Box t}(\widetilde{T})\cdot \delta+O(\delta^2).
\end{array}
\end{equation}
Replacing \eqref{transv3} into  \eqref{transv2}, the arbitrariness of $\delta$ allows us to deduce the condition
$$L\left( \widetilde{T},\widetilde{y}(\widetilde{T}),\frac{\Box \widetilde{y}}{\Box t}(\widetilde{T}) \right)= \frac{\partial L}{\partial v}[\widetilde{y}](\widetilde{T})\cdot \left(\frac{\Box \widetilde{y}}{\Box t}(\widetilde{T})- \frac{\Box \psi}{\Box t}(\widetilde{T})\right).$$
\end{proof}

%
%
%
%
%

\subsection{Higher-order variational problems}

In this section we  consider the following higher-order functional $\mathcal{J}$

\begin{equation}\label{functional J}
\mathcal{J}[y,T]=\displaystyle \int_a^T L\left(t, y(t), \frac{\Box y}{\Box t}(t), \cdots, \frac{\Box^n y}{\Box t^n}(t)\right) \, dt
\end{equation}

\noindent defined in the class

$$\mathcal{E}= \{y \in H^\alpha (I^n, \mathbb{R}): y(a)=y_a, \frac{\Box y}{\Box t}(a)=y^1_a , \cdots, \frac{\Box^{n-1} y}{\Box t^{n-1}}(a)=y^{n-1}_a
\wedge y \in C^n_\Box([a,b], \mathbb{R}) \},$$

\noindent where the Lagrangian $L=L(t, y, v_1, v_2, \cdots, v_n):[a,b]\times \mathbb{R} \times \mathbb{C}^n  \rightarrow \mathbb{C}$ is of class $C^1$,
$T\in \mathbb{R}$ is such that $a\leq T\leq b$,
and $y_a\in \mathbb{R}$, $y^1_a, \cdots, y^{n-1}_a \in \mathbb{C}$ are given fixed numbers.

\begin{Definition}\label{higher-order scale extremal}
We say that $(y,T)$ is a scale extremal of functional $\mathcal{J}$ defined on $\mathcal{E}$
if, for any $\eta \in H^\beta(I^n, \mathbb{R})\cap C^n_{\Box}([a,b], \mathbb{R})$ such that
$\eta(a)=\frac{\Box \eta}{\Box t}(a)= \cdots= \frac{\Box^{n-1} \eta}{\Box t^{n-1}}(a)=0$, and $\delta \in \mathbb{R}$,
$$
\frac{d}{d\varepsilon}\left.\mathcal{J}[y + \varepsilon \eta, T + \varepsilon \delta ]\right|_{\varepsilon = 0}=0.
$$
\end{Definition}

In order to simplify notations we will denote
$$[y]^n(t):= \left(t, y(t), \frac{\Box y}{\Box t}(t), \cdots, \frac{\Box^n y}{\Box t^n}(t)\right).$$

\begin{Theorem}[The higher-order scale Euler--Lagrange equation and natural boundary conditions]
\label{E-L+HTC}
Let $\widetilde{T}\in [a,b]$
and $\widetilde{y} \in \mathcal{E}$ be such that
\begin{enumerate}
\item $\displaystyle\frac{\partial L}{\partial v_i}[\widetilde{y}]^n \in H^\alpha(I^n,\mathbb{C})$ for all $i=1, 2, \cdots, n$,
\item for all $i=1, 2, \cdots, n$ and $k=0, 1,\cdots, i-1$,
$$\lim_{h \to 0} \int_a^{\widetilde{T}}\left( \frac{\Box_h}{\Box t}\left(\frac{\Box^k}{\Box t^k}(\frac{\partial L}{\partial v_i}[\widetilde{y}]^n)\cdot \frac{\Box^{i-k-1} \eta}{\Box t^{i-k-1}} \right)(t) \right)_E \, dt =0,$$
for all $\eta \in H^\beta(I^n,\mathbb{R}) \cap C^n_{\Box}([a,b], \mathbb{R})$ such that $\eta(a)=0, \frac{\Box \eta}{\Box t}(a)=0, \cdots, \frac{\Box^{n-1} \eta}{\Box t^{n-1}}(a)=0 $.
\end{enumerate}
The pair $(\widetilde{y},\widetilde{T})$ is a scale extremal of functional $\mathcal{J}$  on the class $\mathcal{E}$
if and only if the following conditions hold:
\begin{enumerate}
\item $\displaystyle\frac{\partial L}{\partial y}[\widetilde{y}]^n(t)
+ \sum_{i=1}^n (-1)^i \frac{\Box^i}{\Box t^i}\left(\frac{\partial L}{\partial v_i}\right)[\widetilde{y}]^n(t)=0$ for all $t\in [a,\widetilde{T}]$;
\item  $\displaystyle \sum_{k=i}^n (-1)^{k-i} \frac{\Box^{k-i}}{\Box t^{k-i}} \left(\frac{\partial L}{\partial v_k}\right) [\widetilde{y}]^n(\widetilde{T}) =0, \quad \forall i=1, 2, \cdots, n;$
\item $ \displaystyle L[\widetilde{y}]^n(\widetilde{T}) =0$.
\end{enumerate}
\end{Theorem}

\begin{proof}
Suppose that $(\widetilde{y},\widetilde{T})$ is a scale extremal of functional $\mathcal{J}$  on the class $\mathcal{E}$.
For variation functions, we consider $(\widetilde{y} + \varepsilon \eta, \widetilde{T} + \varepsilon \delta)$, with $\eta$ be such that
$$\eta(a)=\frac{\Box \eta}{\Box t}(a)= \cdots= \frac{\Box^{n-1} \eta}{\Box t^{n-1}}(a)=0.$$

\noindent
Using the definition of scale extremal, we conclude that
$$
\displaystyle \int_a^{\widetilde{T}} \left( \frac{\partial L}{\partial y}[\widetilde{y}]^n (t) \cdot \eta(t)+
\sum_{i=1}^n \frac{\partial L}{\partial v_i}[\widetilde{y}]^n (t)\cdot \frac{\Box^i \eta}{\Box t^i}(t) \right) dt
+  \displaystyle L[\widetilde{y}]^n (\widetilde{T})\cdot \delta =0.
$$
\noindent
Applying the integration by parts formula, we get

\begin{equation*}
\begin{split}
& \displaystyle \int_a^{\widetilde{T}} \left( \frac{\partial L}{\partial y}[\widetilde{y}]^n (t)
+
\sum_{i=1}^n (-1)^i \frac{\Box^{i}}{\Box t^{i}}(\frac{\partial L}{\partial v_i})[\widetilde{y}]^n (t)\right) \cdot \eta(t) dt
\\
& + \sum_{i=1}^n  \left(\frac{\partial L}{\partial v_i}[\widetilde{y}]^n (\widetilde{T})\cdot \frac{\Box^{i-1} \eta}{\Box t^{i-1}}(\widetilde{T})  +
\sum_{k=1}^{i-1}(-1)^k \frac{\Box^{k}}{\Box t^{k}}(\frac{\partial L}{\partial v_i})[\widetilde{y}]^n (\widetilde{T})\cdot \frac{\Box^{i-1-k} \eta}{\Box t^{i-1-k}}(\widetilde{T}) \right)
+  \displaystyle L[\widetilde{y}]^n (\widetilde{T})\cdot \delta =0.
\end{split}
\end{equation*}

Considering $\delta=0$ and $\eta(\widetilde{T})=\frac{\Box \eta}{\Box t}(\widetilde{T})= \cdots= \frac{\Box^{n-1} \eta}{\Box t^{n-1}}(\widetilde{T})=0$ we obtain the higher-order scale Euler-Lagrange equation.
Similarly as done in the proof of Theorem \ref{E-L+TC}, for appropriate variations, we obtain the pretended natural boundary conditions.

\end{proof}

Clearly, all the results presented in Subsection 3.1 can be generalized for higher-order variational problems.


\section{Conclusions}

In the present paper we study variational problems when the dynamic of the trajectories are nondifferentiable. To overcome this situation, we considered the scale derivative as presented by Cresson and Greff in \cite{CressonGreff} and \cite{Cresson1}, which has shown some applications in physics, e.g., trajectories of quantum mechanics, fractals and scale-relativity theory. The main aim was to find necessary and sufficient conditions that a pair $(y,T)$ must satisfy in order to be an extremal of a given functional, where $y$ is the trajectory and $T$ the end-time of the integral. We considered the existence or not of boundary conditions on the initial and end time, as well with higher-order scale derivatives.


\section*{Acknowledgments}

This work was supported by Portuguese funds through the CIDMA - Center for Research and Development in Mathematics and Applications, and the Portuguese Foundation for Science and Technology ("FCT –- Funda\c{c}\~{a}o para a Ci\^{e}ncia e a Tecnologia"), within project PEst-OE/MAT/UI4106/2014.




\end{document}